\documentclass[11pt,reqno]{amsart}
\usepackage{graphicx}
\usepackage{float}
\usepackage{color}
\usepackage{tikz}
\usepackage{soul,xcolor}

\usepackage{amsfonts}
\usepackage[caption = false]{subfig}
\usepackage{amsmath}
\usepackage{graphics}
\usepackage{float}
\usepackage{enumerate}
\usepackage{mathtools}
\usepackage{amsthm}
\usepackage[english]{babel}
\usepackage{amsfonts,amssymb}
\usepackage{mathrsfs}
\usepackage[utf8x]{inputenc}\usepackage[english]{babel}
\usepackage{soul}
\usepackage{cite}
\usepackage{relsize}
\numberwithin{equation}{section}
\newcommand{\ds}{\displaystyle}
\newcommand{\nni}{\noindent}
\newcommand{\mR}{{\mathbb{R}}}
\newcommand{\mC}{{\mathbb{C}}}

\newcommand{\mD}{{\mathbb{D}}}
\newcommand{\mN}{{\mathbb{N}}}

\newtheorem{theorem}{Theorem}[section]
\newtheorem{corollary}[theorem]{Corollary}
\newtheorem{lemma}[theorem]{Lemma}

\theoremstyle{definition}

\theoremstyle{remark}

\numberwithin{equation}{section}

\begin{document}
	
	\title[On the Convexity of the Bernardi Integral Operator]{On the Convexity of the Bernardi Integral Operator}
	
	\author[Johnny E. Brown]{Johnny E. Brown}
	\address{Department of Mathematics,
Purdue University, West Lafayette, Indiana 47907 USA}
	\email{brown00@purdue.edu}

	\subjclass[2010]{30C45, 30C80}
	
	\keywords{Starlike, convex, integral operator, univalent functions}
    
\begin{abstract}
We prove that the Bernardi Integral Operator maps  certain classes of bounded starlike functions into the class of convex functions, improving the result of Oros and Oros. We also present a general unified method for investigating various other integral operators that preserve many of the previously studied subclasses of univalent and p-valent functions.

\end{abstract}
\maketitle

\section{Introduction and Main Results} 

Let ${\mathcal{H}}(\mD)$ denote the class of functions analytic in the unit disk $\mD$ and let ${\mathcal{H}}_0(\mD)$ be those functions $f\in {\mathcal{H}}(\mD)$ normalized by $f(0)=0$ and $f'(0)=1$. Let $S$ be those functions $f\in {\mathcal{H}}_0(\mD)$ which are univalent. The well-known subclasses of $S$ consisting of univalent starlike and convex functions, denoted by $S^*$ and $K$, satisfy 
$$ \Re e\left\{\frac{zf'(z)}{f(z)}\right\}>0 \quad {\rm{and}}\quad \Re e\left\{1+\frac{zf''(z)}{f'(z)}\right\}>0, $$
for all $z\in\mD$, respectively. It is known that $K\subset S^*\subset S$. J. W. Alexander [1] was the first to study integral operators on $S$. There are many integral operators $T: {\mathcal{H}}_0(\mD)\longrightarrow {\mathcal{H}}(\mD)$ which preserve $S^*$ and $K$. For example the

$${\rm{Alexander\, Transform}}:\quad {\mathcal{A}}(f)(z)=\int^z_0\frac{f(\zeta)}{\zeta}\, d\zeta\quad $$
and
$${\rm{Libera\, Transform}}:\quad {\mathcal{L}}(f)(z)=\frac{2}{z}\int^z_0f(\zeta)\, d\zeta\quad $$
satisfy ${\mathcal{A}}(S^*)\subset S^*, {\mathcal{L}}(S^*)\subset S^*, {\mathcal{A}}(K)\subset K$, and  ${\mathcal{L}}(K)\subset K$. It is also known that  there are many other integral operators which preserve $S^*$ [14]. Typically, the proofs use the method of differential subordination developed by Mocanu and Miller [15]. Alexander proved that ${\mathcal{A}}(S^*)\subset K$ and it is known that for the Libera Transform, ${\mathcal{L}}(S^*)\not\subset K$. Hence it is of interest to ask which subclasses of $S^*$ are mapped into the smaller class $K$ under the Libera Transform. Such a result was recently obtained by Oros and Oros [16] as follows.   For $0<M<1$, if $S_M$ is the well-known subclass of $S^*$ defined by 
\[ S_M=\left\{ f\in H_0(\mD):\, \left|\frac{zf'(z)}{f(z)}-1\right|<M,\, z\in\mD\right\}, \tag{1.1} \label{eq:special}    \]
then for the Libera Transform they proved that ${\mathcal{L}}(S_M)\subset K$,  provided $0<M<M^*$, where $M^*\approx 0.4128\cdots $ is the smallest positive root of 
$$7M^8+14M^7+48M^6+30M^5+67M^4+18M^3+2M-8=0.$$

\nni It is an open question to find the largest possible $M_{\mathcal{L}}$ such that ${\mathcal{L}}(S_{M_{\mathcal{L}}})\subset K$. In addition, they produced an example to show that $M_{\mathcal{L}}\le \frac{3}{5}$ and hence  their results show that the best possible $M_{\mathcal{L}}$ must satisfy
\[M^*\le M_{\mathcal{L}}\le \frac{3}{5} \tag{1.2} \label{eq:special}    \]
\nni In this paper we first prove that there is a one-parameter family of integral operators, which includes both the Alexander and Libera transforms, that preserves $S_M$ and will also map certain  $S_M$ into $K$. As a special case of this result we  improve both the upper and lower bounds on the best possible $M_{\mathcal{L}}$ for the Libera Transform.

For $c>-1$, the Bernardi Integral Operator (or Bernardi Transform) of $f\in {\mathcal{H}}_0(\mD)$ is defined by  
\[ {\mathcal{B}}_c(f)(z)=\frac{1+c}{z^c}\int^z_0\,\zeta^{c-1} f(\zeta)\,d\zeta. \tag{1.3} \label{eq:special}    \]

\nni If $F(z)={\mathcal{B}}_c(f)(z)$, we see that $F$ is the unique solution to the linear differential equation with forcing function $f$:
$$zF'(z)+cF(z)=(1+c)f(z),$$
where $F(0)=0$.

\nni Observe that ${\mathcal{B}}_0(f)={\mathcal{A}}(f)$ and ${\mathcal{B}}_1(f)={\mathcal{L}}(f)$. It has been shown that ${\mathcal{B}}_c(S^*)\subset S^*$ (for example [20] and [14]). We prove that  ${\mathcal{B}}_c(S_M)\subset S_M$ for $0<M<1$ (Lemma 2.3), and one of our main results is that ${\mathcal{B}}_c$ maps certain $S_M$ into the smaller class $K$.  More precisely:\medskip

\begin{theorem} Let  $c=0, 1, 2, \cdots$. If $f\in S_M$ and 
$$0<M<M_c\equiv\sqrt{c^2+1}-c,$$
then ${\mathcal{B}}_c(f)$ is convex. Hence ${\mathcal{B}}_c(S_{M_c})\subset K$.
\end{theorem}

\medskip

\nni Note that when $c=0$ we have ${\mathcal{B}}_0(S^*)={\mathcal{A}}(S^*)\subset K$ and the theorem is trivial in this case. However, the special case for the Libera Transform, $c=1$, is of particular interest which we consider in detail in Section $4$. The proof of the above theorem gives rise to several other related results as follows. \medskip

If $f, g\in {\mathcal{H}}(\mD)$, then $f$ is subordinate to $g$, $f(z)\prec g(z)$, if there exists $\omega (z)$ such that $f(z)=g(\omega(z))$,  where $\omega\in {\mathcal{H}}(\mD)$, $\omega(0)=0$ and $|\omega(z)|<1$. It is known that if $g(z)$ is univalent, then $f(z)\prec g(z)$ if and only if $f(0)=g(0)$ and $f(\mD)\subset g(\mD)$.\medskip

There are some subclasses of starlike functions which are preserved under the Bernardi Transform as well as other linear and nonlinear integral operators (for example [14]). We can prove a very general result which contains some of these previous results, as well as several new ones, and in addition provide a unified method to investigate many subclasses of $S^*$:

\begin{theorem} Let $\Psi (z)$ be any convex univalent function in $\mathcal{H}(\mD)$ with $\Psi(0)=1$ and $ \Re e \,\Psi(z)>0$ for all $z\in\mD$. If  $${\mathcal{F}}=\left\{f\in \mathcal{H}_0(\mD):\, \frac{zf'(z)}{f(z)}\prec \Psi(z)\right\},$$
then ${\mathcal{B}}_c({\mathcal{F}})\subset {\mathcal{F}}$, for $c=0, 1, 2, \cdots$. Hence the class ${\mathcal{F}}$ is preserved by the Bernardi Transform for $c=0, 1, 2, \cdots$. 
\end{theorem}
\medskip

\nni When $\Psi(z)$ is starlike, the class ${\mathcal{F}}$ is referred to as the Ma-Minda Type (MMT) starlike class [13]. A small sampling of special cases where  the Bernardi Transform preserves ${\mathcal{F}}$ include the following - most of which are apparently new results:
\begin{itemize}
\item[(a)]\, If $\ds\Psi(z)=1+\left(\frac{2(1-\alpha)z}{1-z}\right)$, then $\ds {\mathcal{F}}=S^*(\alpha)=\left\{ f\in{\mathcal{H}}_0:\, \Re e \left(\frac{zf'(z)}{f(z)}\right)>\alpha   \right\}$ is the class of starlike functions of order $\alpha$, where $\, 0\le\alpha<1$. \medskip

\item[(b)] If $\ds\Psi(z)=\left(\frac{1+z}{1-z}\right)^{\!\beta}$, then ${\mathcal{F}}=SS^*(\beta)$ is the class of strongly starlike functions of order $\beta$ where $\ds \left|\arg\left(\frac{zf'(z)}{f(z)}\right)\right|<\frac{\beta\pi}{2}$, where $0<\beta\le 1$.\medskip

\item[(c)] If $\ds\Psi(z)=\left(\frac{1+Az}{1+Bz}\right)$ , then ${\mathcal{F}}$ is the  Janowski generalized class of starlike functions $S^*[A,B]$, where $\ds \frac{zf'(z)}{f(z)}\prec  \left(\frac{1+Az}{1+Bz}\right)$ and $-1\le B<A\le 1$ (This particular case was already proved in [6]).\medskip

\item[(d)] If $\Psi(z)=\cos z$ and $\Psi(z)=1+\sin z$,  then ${\mathcal{F}}$ is the subclass of starlike functions $S^*_c$ and $S^*_s$, respectively,  studied in [5, 7].\medskip

 \item[(e)]  If $\Psi (z)$ is the  fractional linear map  $\ds\Psi(z)=\frac{(\beta^2+\alpha-\alpha^2)z+\beta}{z(1-\alpha)+\beta}$, then ${\mathcal{F}}$ is the class of all bounded functions $f\in{\mathcal{H}}_0(\mD)$ with $\ds\left|\frac{zf'(z)}{f(z)}-\alpha\right|<\beta$, where  $0<\beta\le\alpha\le 1$, which has also been studied by several authors.
 
 \item[(f)]  If $\Psi (z)$ is the  fractional linear map  $\ds\Psi(z)=\frac{1-z}{1+ze^{2i\beta}}$, where $|\beta|<\frac{\pi}{2}$, then ${\mathcal{F}}$ is the class of spirallike functions $f\in{\mathcal{H}}_0(\mD)$ with $\ds \Re e\left\{e^{i\beta}\left(\frac{zf'(z)}{f(z)}\right)\right\} >0$, introduced and studied  by  L. $\rm{{\breve{S}}pa{\breve{c}}ek}$ [21] in 1933.
\end{itemize}
\bigskip

\section{Preliminaries}

\nni Our first result generalizes a result of Libera[11] and was proved  previously by the author in [4]. We include the proof here for completeness.\medskip

\begin{lemma} 
Suppose $g\in {\mathcal{H}}(\mD)$ and maps $\mD$ onto a finitely-sheeted region starlike with respect to the origin. If $\Phi\in {\mathcal{H}}(\mD)$ is a univalent convex mapping, 
$\ds k\in {\mathcal{H}}(\mD), k(0)=g(0)=0, \,\frac{k(0)}{g(0)}=\Phi(0)$ and  $\ds \frac{k'(z)}{g'(z)}\prec\Phi(z)$, then $\ds \frac{k(z)}{g(z)}\prec\Phi(z)$.
\end{lemma}

\begin{proof} 
Fix $z_0\in\mD$ and let $w_0=g(z_0)$. Choose a branch of $g^{-1}$ defined on one of the starlike sheets $\Omega$ so that $K(w)=k\circ g^{-1}(w)$ is analytic in $\Omega$. Since $\Omega$ is starlike, the ray from $0$ to $w_0$ lies in $\Omega$, so if we let $x(t)$, for $0\le t\le 1$, be the preimage of this ray, since $\Phi(\mD)$ is convex we conclude
$$\begin{array}{rl}
\ds\frac{K(w_0)}{w_0}&=\ds\int^1_0\,K'(tw_0)\,dt=\ds\int^1_0\,\frac{k'(x(t)) }{g'(x(t)) }\, dt\\ & \\

&\ds \subset\, {\rm{convex\,\, hull\,\, of}}\,\, \left\{\frac{k'(z) }{g'(z)}:\, z\in\mD \right\}\subset \Phi(\mD).
\end{array}$$
Hence $\ds\frac{k(z_0)}{g(z_0)}\subset\Phi(\mD)$ and $\ds \frac{k(0)}{g(0)}=\Phi(0)$ so the result follows by univalence of $\Phi$.  
\end{proof}

\begin{lemma} 
If $f\in S^*$, then the function $\ds\int^z_0\, \zeta^{c-1}f(\zeta)\, d\zeta$ is also starlike and $(1+c)-$valent for $c=0, 1, 2, \cdots$.
\end{lemma}

\begin{proof} 
Let $g(z)$ and $k(z)$ be defined by

\[ g(z)=z^cf(z) \tag{2.1} \label{eq:special}    \]

\[ k(z)=\int^z_0\, \zeta^{c-1}f(\zeta)\, d\zeta \tag{2.2} .\label{eq:special}    \]

\noindent Hence $zk'(z)=g(z)$ and  (2.1) also gives
$$\frac{zg'(z)}{g(z)}=\frac{zf'(z)}{f(z)}+c.$$
Since $f$ is starlike and $c\ge 0$, it follows that $g$ is also starlike and since $f$ is univalent, the function $g$ is $(1+c)-$valent.

Next, we see that 
$$\frac{k'(z)}{g'(z)}=\frac{z^{c-1}f(z)}{z^cf'(z)+cz^{c-1}f(z)}=\left[\left(\frac{zf'(z)}{f(z)}\right)+c\right]^{-1}.$$

\noindent Thus $\ds \Re e\left\{\frac{k'(z)}{g'(z)} \right\}>0, \, k(0)=g(0)=0$ and $\ds\frac{k(0)}{g(0)}=\frac{k'(0)}{g'(0)}=\frac{1}{1+c}$ and we conclude that

$$\frac{k'(z)}{g'(z)}\prec \frac{1}{1+c}\left(\frac{1+z}{1-z}\right).$$
We can apply Lemma 2.1 with $\ds \Phi(z)=\frac{1}{1+c}\left(\frac{1+z}{1-z}\right)$ to conclude that
$$\frac{k(z)}{g(z)}\prec \frac{1}{1+c}\left(\frac{1+z}{1-z}\right)$$
and in particular, $\ds\Re e\left\{ \frac{k(z)}{g(z)} \right\}>0$. Finally, we note that $\ds\frac{k(z)}{g(z)}=\frac{k(z)}{zk'(z)}$ and hence $\ds k(z)=\int^z_0\,\zeta^{c-1} f(\zeta)\,d\zeta$ is starlike. Moreover we conclude that $k$ is $(1+c)-$valent since 
$$ \frac{1}{2\pi}\int^{2\pi}_0\Re e\left\{\frac{re^{i\theta}k'(re^{i\theta})}{k(re^{i\theta})}\right\}\, d\theta=\Re e\left\{\frac{g(0)}{k(0)}\right\}=(1+c).$$

\end{proof}

\bigskip

The next result show that the subclasses $S_M$ are preserved under the Bernardi Transform.\medskip

\begin{lemma}
 If $f\in S_M$,  where $0<M<1$, then ${\mathcal{B}}_cf\in S_M$ for each $ c=0,1,2,3, \cdots$. Hence ${\mathcal{B}}_c(S_M)\subset S_M$.
 \end{lemma}

\begin{proof} Observe that  by (1.1), we have $f\in S_M$ if and only if 
\[ \frac{zf'(z)}{f(z)}\prec 1+Mz. \tag{2.3} \label{eq:special}    \]
For convenience, let $\ds F(z)={\mathcal{B}}_c(f)(z)=\frac{1+c}{z^c}\int^z_0\, \zeta^{c-1}f(\zeta)\,d\zeta$. Since $\ds\frac{zf'(z)}{f(z)}\prec (1+Mz)$ we have

$$\frac{\big(z^cf(z)\big)' }{\big(\int^z_0\,\zeta^{c-1}f(\zeta)\,d\zeta \big)' } = \frac{z^cf'(z)+cz^{c-1}f(z) }{ z^{c-1}f(z)  }=\frac{zf'(z)}{f(z)}+c\prec (1+Mz)+c.$$

\nni By Lemma 2.2,    $\ds g(z)=\int^z_0\,\zeta^{c-1}f(\zeta)\,d\zeta $ is starlike and $(1+c)-$valent, $k(z)=z^cf(z)$ is starlike and $\ds \frac{k(0)}{g(0)}=\frac{k'(0)}{g'(0)}=(1+c)$. We  now apply Lemma 2.1 this time with $\Phi(z)=\left(1+Mz\right)+c$, to conclude

\[\frac{z^cf(z)}{\int^z_0\,\zeta^{c-1}f(\zeta)\,d\zeta  } \prec \left(1+Mz\right)+c \tag{2.4} \label{eq:special}  \]

\nni Observe that since $$\frac{z^cf(z)}{\int^z_0\,\zeta^{c-1}f(\zeta)\,d\zeta  } =\frac{zF'(z)}{F(z)}+c,$$ (2.4) gives

$$\frac{zF'(z)}{F(z)}+c\prec (1+Mz)+c.$$
Thus $\ds\frac{zF'(z)}{F(z)}\prec (1+Mz)$ and so $F\in S_M$.
\end{proof}

\bigskip

\section{Proofs of Main Results}

\begin{proof} {\bf{\underline{Theorem 1.1}}}:

Let $\ds  f\in S_M$ and let $\ds F(z)={\mathcal{B}}_c(f)(z)=\frac{1+c}{z^c}\int^z_0\,\zeta^{c-1}f(\zeta)\, d\zeta, $ its Bernardi Transform.  If $\ds p(z)=\frac{z F'(z)}{F(z)}$, then a calculation shows $f\in S_M$ if and only if

 \begin{equation}\left|\frac{zp'(z)}{p(z)+c}+p(z)-1\right|<M .\tag{3.1} \label{eq:special}\end{equation}

\nni Define $\ds \Lambda(z)\equiv 1+\frac{zF''(z)}{F'(z)}$ and note that 

$$\Lambda(z)=1+\frac{zF''(z)}{F'(z)}=\left\{\frac{zp'(z)}{p(z)}+p(z)\right\}.$$

\nni In order to prove Theorem 1.1, it suffices to prove that $\ds \Re e\big\{\Lambda (z)\Big\}>0$ for all $z\in\mD$, with 
$c$ and $M$ satisfying $$\ds 0<M<\sqrt{c^2+1}-c.$$
This implies the more useful inequality 
\begin{equation}  \Big[(1-c)(1+M)\Big]>-c+M(c+1)+M^2.\tag{3.2} \label{eq:special}\end{equation}

\nni If $c=0$ then (3.1) implies $F$ is convex for all $0<M<1$. Hence we need only consider fixed $c=1, 2, 3,\cdots $ and then $\sqrt{c^2+1}-c\le \sqrt{2}-1$. \medskip

\nni A calculation shows that (3.1) is equivalent to
  
 \begin{equation}\left|\Lambda (z)+(c-1)-\frac{c}{p(z)}\right|<M \left|   1+\frac{c}{p(z)}\right|. \tag{3.3} \label{eq:special}\end{equation}

\nni By Lemma 2.3, $\ds\frac{1}{p(z)}\prec \frac{1}{1+Mz}$ and so $\ds\frac{1}{p(z)}$ lies inside or on the closed disk $\left| z-a\right|\le r$, where $\ds a=\frac{1}{1-M^2}$ and $\ds r=\frac{M}{1-M^2}$ . Fix $z_0\in\mD$, thus 

\[  \frac{1}{p(z_0)}=a+\epsilon r e^{it},\quad 0\le\epsilon\le 1,\,\, -\pi<t\le\pi. \tag{3.4} \label{eq:special}     \]

\nni From (3.3) we get

 \begin{equation}\Big\{\Lambda (z_0)\Big\}+(c-1)-\frac{c}{p(z_0)}=M \delta e^{i\theta}\left(   1+\frac{c}{p(z_0)}\right) \tag{3.5} \label{eq:special}\end{equation}
 and hence 

\begin{equation}\Re e\Big\{\Lambda (z_0)\Big\}=(1-c)+\Re e\left\{\frac{c}{p(z_0)}\right\}+\Re e\left\{ M\delta e^{i\theta}\left(1+\frac{c}{p(z_0)}\right)\right\} \tag{3.6} \label{eq:special}\end{equation}
where $0\le\delta<1,\, -\pi<\theta\le \pi$. For fixed but arbitrary $\delta$ and $\theta$, and letting $\ds\frac{1}{p(z_0)}=w$,  the right hand side  of (3.6) is harmonic as a function of $w$ and hence is minimized for $|w|=r$. Thus we need only minimize (3.6) for $\epsilon=1$ and hence

$$ \Re e\Big\{\Lambda (z_0)\Big\}\ge (1-c)+\Re e\Big\{c(a+re^{it})\Big\}+M\delta\,\Re e\Big\{e^{i\theta}\left((1+ca)+cre^{it}\right)\Big\} $$
Thus, without loss of generality,    

\[ \Re e\Big\{\Lambda (z_0)\Big\}\ge (1-c)+c(a+r\cos t)+M\delta\Big\{(1+ca)\cos\theta+cr\cos (\theta+t)   \Big\} \tag{3.7} \label{eq:special}\]

\nni Note that if the term $M\delta\,\big\{(1+ca)\cos\theta+cr\cos (\theta+t)   \big\}\ge 0 $, then using (3.2), the above inequality (3.7) gives

$$\begin{array}{rl}\ds \Re e\Big\{\Lambda (z_0)\Big\}&\ge (1-c)+c(a+r\cos t)\\
&\ge (1-c)+c(a-r)\\
&\ds =(1-c)+\frac{c}{1+M}\\
&=(1+M)^{-1}\Big\{\big[(1-c)(1+M)\big]+c   \Big\}\\
&>(1+M)^{-1}\Big\{M(c+1)+M^2  \Big\}\\ &>0.   \end{array}$$
While if $M\delta\,\big\{(1+ca)\cos\theta+cr\cos (\theta+t)   \big\}< 0 $, then to minimize $\Re e\Big\{\Lambda (z_0)\Big\}$ we may assume $\delta=1$. If we let $x=\cos\theta$ and $y=\cos t$, then it suffices to prove that  $\Re e\Big\{\Lambda (z_0)\Big\}\ge\phi(x,y)>0$, where

\begin{equation} \phi(x,y)=(1-c)+ c(a+ry)+M(1+ca)x+Mcr\Big\{xy-\sqrt{(1-x^2)(1-y^2)}\,\Big\} \tag{3.8} \label{eq:special}
    \end{equation}
We want to minimize $\phi(x,y)$ over the square $\Omega: [-1, 1]\times[-1, 1]$. We first show $\phi(x,y)>0$ on $\partial \Omega$.\medskip
 
\nni When $y=-1$, with $\ds a=\frac{1}{1-M^2}$ and $\ds r=\frac{M}{1-M^2}$, using (3.2) we get
$$\begin{array}{rl}
\ds \phi(x,-1)&\ds =(1-c)+c(a-r)+M\big(1+c(a-r)\big)x \\
&\ge (1-c)+c(a-r)-M\big(1+c(a-r)\big)  \\
&\ds =(1+M)^{-1}\Big\{\big[(1-c)(1+M)\big]+c-M(1+M+c)\Big\} \\
& >(1+M)^{-1}\Big\{-c+M(c+1)+M^2 +c-M(1+M+c)    \Big\}\\
&=0.
\end{array}$$
Similarly, using (3.2) as above,  if $x=-1$:
    $$\begin{array}{rl}
\ds \phi(-1,y)&\ds =(1-c)+c(a+ry)-M(1+ca)-Mcry \\
&\ge (1-c)+c(a-r)-M(1+ca)+Mcr  \\
&= (1-c)+c(a-r)-M\big(1+c(a-r)\big)  \\
&\ds =(1+M)^{-1}\Big\{\big[(1-c)(1+M)\big]+c-M(1+M+c)\Big\} \\
& >(1+M)^{-1}\Big\{-c+M(c+1)+M^2 +c-M(1+M+c)    \Big\}\\
&=0.
\end{array}$$

  \nni When $y=1$ we obtain  
   $$\begin{array}{rl}
\ds \phi(x,1)&\ds =(1-c)+c(a+r)+M (1+ca)x+Mcrx \\
&\ge (1-c)+c(a+r)-M\big(1+c(a+r)\big)  \\
&=(1-M)^{-1}\big[(1-c)(1-M)+c-M\left(1-M+c\right)  \big] \\
&=(1-M)\\
&>0.

\end{array}$$

\nni The final boundary segment when $x=1$ and using (3.2) once again gives
$$\begin{array}{rl}
\ds \phi(1,y)&\ds =(1-c)+c(a+ry)+M(1+ca) +Mcry \\
&\ge (1-c)+c(a-r)+M(1+ca)-Mcr  \\
&\ds =(1+M)^{-1}\Big\{\big[(1-c)(1+M)\big]+c+M\big\{(1+M)+c\big\}  \Big\} \\
& >(1+M)^{-1}\Big\{\big[-c+M(c+1)+M^2 \big]+c+M(1+M+c)    \Big\}\\
&=2M(1+M)^{-1}(1+c+M)\\
&>0.\\

\end{array}$$

\nni Hence $\Re e\big\{\Lambda (z_0)\big\}\ge \phi(x,y)>0$ on $\partial \Omega$.  Next, for possible critical points, it follows that $\nabla \phi ={\bf{0}}$ at $(x,y)$ when

$$ \frac{\partial \phi}{\partial x}=M(1+ca)+Mcry+\frac{Mcrx(1-y^2) }{\sqrt{(1-x^2)(1-y^2)} } =0$$

\[ \frac{\partial \phi}{\partial y}=cr(1+Mx)+\frac{Mcry(1-x^2) }{\sqrt{(1-x^2)(1-y^2)} } =0. \]

\nni We need to show $\phi(x,y)>0$. The last equation gives 
$$ \sqrt{(1-x^2)(1-y^2)}=- \frac{My\left(1-x^2\right)}{1+Mx}$$
and thus (3.8) shows that $\phi(x,y)$ at any critical point $(x,y)$ has the form:
$$ \phi(x,y)=(1-c)+ c(a+ry)+M(1+ca)x+Mcrxy+Mcr\left( \frac{My(1-x^2) }{1+Mx  } \right). $$

\nni For convenience, let $\psi(x,y)=(1+Mx)\phi(x,y)$, so we have 
 
\[ \psi(x,y)=(1+Mx)\Big\{(1-c)+ c(a+ry)+M(1+ca)x+Mcrxy\Big\}+M^2cry(1-x^2).\tag{3.9} \label{eq:special}\]

\nni It follows that $\ds\frac{\partial\psi}{\partial y}=cr(1+2Mx+M^2)>0$ and hence 
$$\psi(x,y)\ge\psi(x,-1)=(1+Mx)\Big\{(1-c)+c(a-r)+M\big[1+c(a-r)\big]x\Big\}-M^2cr(1-x^2).$$  
A further calculation then gives

$$\begin{array}{rl} 
\ds\frac{1}{M}\,\psi'(x,-1)&= (1-c)+c(a-r)+M\big[1+c(a-r)\big]x+(1+Mx)\big[1+c(a-r)\big]+2Mcrx\\
&\ge  (1-c)+c(a-r)-M\big[1+c(a-r)\big]+(1-M)\big[1+c(a-r)\big]-2Mcr\\
&=\ds  (1-c)+c(a-r)+(1-2M)\big[1+c(a-r)\big]-2Mcr\\
&=\ds  (1-c)+\frac{c}{1+M}+(1-2M)\left[1+\frac{c}{1+M}\right]-2Mc\left(\frac{M}{1-M^2}\right)\\
&\ds =\frac{(1-M)\Big[(1+M)(1-c)\Big]+c(1-M)+(1-2M)\big[(1-M^2)+c(1-M)\big]-2M^2c }{(1-M^2)}\\
\end{array}.$$
Using (3.2) and the fact that $0<M< \sqrt{2}-1$, the last result implies
$$\begin{array}{rl}
\ds\frac{1}{M}\,\psi'(x,-1)&\ds >\frac{(1-M)\Big[-c+M(c+1)+M^2\Big]+c(1-M)+(1-2M)\big[(1-M^2)+c(1-M)\big]-2M^2c }{(1-M^2)}\\
&\ds =\frac{(1-M)(1-M^2)+c(1-2M-M^2) }{(1-M^2)}\\
&>0.
\end{array}$$

\nni Therefore we see that $\psi(x,-1)$ increases with $x$. Finally, from (3.9) and  using (3.2) again, we get

$$\begin{array}{rl}
\ds \psi(x,y)&\ge\psi(x,-1) \\ 
&\ge\psi(-1,-1)\\
&= (1-M)\Big\{(1-c)+c(a-r)-M\big[1+c(a-r)\big]\Big\}  \\
&\ds =\left(\frac{1-M}{1+M}\right)\Big\{ \big[(1+M)(1-c)\big]+c-M\big[(1+M)+c\big]\Big\}\\
&\ds>\left(\frac{1-M}{1+M}\right)\Big\{ \big[-c+M(c+1)+M^2]+c-M\big[(1+M)+c\big]\Big\}\\
&=0.
\end{array}$$
This now gives $\psi(x,y)>0$, and since $\phi(x,y)=(1+Mx)\psi(x,y)>0$, we conclude that $\Re e\big\{\Lambda (z_0)\big\}\ge \phi(x,y)>0$  at any possible critical points as well and the proof is complete. 
\end{proof}

\medskip

\begin{proof} {\bf{\underline{Theorem 1.2}}}:

Every $f\in {\mathcal{F}}$ is starlike and hence by Lemma 2.2,  $\ds g(z)=\int^z_0\zeta^{c-1}f(\zeta)\,d\zeta\in S^*$ and $(1+c)-$valent. Let $\ds F(z)={\mathcal{B}}_c(f)(z)=\frac{1+c}{z^c}\int^z_0\, \zeta^{c-1}f(\zeta)\,d\zeta$ and note that
$$\frac{\big(z^cf(z)\big)' }{\big(\int^z_0\,\zeta^{c-1}f(\zeta)\,d\zeta \big)' } =\frac{zf'(z)}{f(z)}+c\prec \Psi(z)+c.$$
Since $\Psi$ is convex, using Lemma 2.1, we get 
$$\frac{zF'(z)}{F(z)}+c=\frac{\big(z^cf(z)\big) }{\big(\int^z_0\,\zeta^{c-1}f(\zeta)\,d\zeta \big) } \prec \Psi(z)+c.$$ 
Hence $\ds\frac{zF'(z)}{F(z)}\prec \Psi(z)$.
\end{proof}

\bigskip

\section{Libera Transform}

The special case when $c=1$ is the Libera Transform $\ds {\mathcal{B}}_1(f)(z)={\mathcal{L}}(f)(z)=\frac{2}{z}\int^z_0f(\zeta)\,d\zeta$ \, is of particular interest and we obtain the following result.\medskip

\begin{theorem}
If $f\in S_M$ and $0<M<\sqrt{2}-1$, then ${\mathcal{L}}(f)$ is convex. Hence ${\mathcal{L}}(S_M)\subset K$ for $0<M<\sqrt{2}-1$.
\end{theorem}

\begin{proof}
Let $c=1$ in Theorem 1.1.
\end{proof}

\medskip

Oros and Oros [16] stated that $f^*(z)=z+\frac{3}{8}z^2\in S_{\frac{3}{5}}$, but ${\mathcal{L}}(f^*) $ is {\em{not}} convex and hence the upper bound for the best possible $M_{\mathcal{L}}$ is $\frac{3}{5}$. Our Theorem 4.1 improves their lower bound for this best constant $M_{\mathcal{L}}$ so that  ${\mathcal{L}}(S_{M_\mathcal{L}})\subset K$ and we can also  improve their upper bound as follows.\medskip

\begin{theorem}
If  ${\mathcal{L}}(S_M)\subset K$, then $ M\le\frac{1}{2}$.\end{theorem}
\medskip

\begin{proof} Assume that ${\mathcal{L}}(S_M)\subset K$, with  $\frac{1}{2}<M<1$. Since   $\ds \mu (n)=\frac{n^2-1}{2n^2-n-1}$ decreases for $n\ge 2$ and $\ds\lim_{n\rightarrow\infty}\mu(n)=\frac{1}{2}$, there exists an integer $m\ge 2$ such that

\[  M>\frac{m^2-1}{2m^2-m-1}. \tag{4.1} \label{eq:special}\]
Let $f_0(z)=z+\lambda z^m$, where $\ds\lambda=\frac{M}{m-1+M}$. It follows that

$$\left|\frac{zf'_0(z)}{f_0(z)}-1\right|=\left|\frac{(m-1)\lambda z^{m-1}}{1+\lambda z^{m-1}}\right|<\frac{(m-1)\lambda}{1-\lambda}=M$$
and hence $f_0\in S_M$. Its Libera Transform is 
$$F_0(z)=\frac{2}{z}\int^z_0 f_0(\zeta)\, d\zeta=z+\alpha z^m,$$
where $\ds \alpha=\frac{2\lambda}{m+1}$, and

\[  1+\frac{zF''_0(z)}{F'_0(z)}=\frac{1+\alpha m^2 z^{m-1}}{1+\alpha m z^{m-1} }. \tag{4.2} \label{eq:special}\]
Because $m\ge 2$, the functions $1+\alpha m^2 z^{m-1}$ and $1+\alpha m z^{m-1}$ cannot both vanish for the same $z$.
Using (4.1) we conclude that 

$$\frac{1}{\alpha m^2}=\frac{(m^2-1)+(m+1)M   }{ 2m^2M }<1$$
and hence from (4.2), the point $\ds z_1=\left\{\frac{1}{\alpha m^2}\right\}^{\frac{1}{m-1}}\in\mD$, and we have $\ds 1+\frac{z_1F''_0(z_1)}{F'_0(z_1)}=0$. Thus $F_0$ cannot be convex and a contradiction arises.\end{proof}
\bigskip

\nni Our work proves that $$\sqrt{2}-1\le M_{\mathcal{L}} \le\frac{1}{2}.$$
While improving the result in [16], this doesn't yet give the best possible $M_{\mathcal{L}}$. However, our work does strongly suggest the following:
\bigskip

\nni {\bf{\underline{Conjecture:}}} If $f\in {\mathcal{H}}_0(\mD)$ and $\ds\left|\frac{zf'(z)}{f(z)}-1\right|<M$, then \, $\ds {\mathcal{L}}\big\{f\big\}\in K \iff M<\sqrt{2}-1.$

\noindent In other words, the constant $M_{\mathcal{L}}=\sqrt{2}-1$ is best possible.
\bigskip

 \nni Finally, it is of interest to note that Oros and Oros [16] also proved that if $f\in S_M$ and $0<M<1$, then its Libera Transform $F\in S_R$, where 
$$ R\le \frac{M-3+\sqrt{M^2+2M+9}}{2}  <M.$$
For fixed, $0<M<1$, it is still also an open problem to find the smallest possible $R$.

\bigskip

\section{Remarks}

The ideas presented here can be adapted and extended in various other directions. For example we can generalize Theorem 1.2 to include nonlinear integral transforms and also $p-$valent functions as follows.  For $p\in\mN$, let ${\mathfrak{A}}(p)$ denote the class of $p-$valent functions  $f\in{\mathcal{H}}(\mD)$  of the form $\ds f(z)=a_pz^p+\sum^{\infty}_{n=p+1}\, a_n z^n$, where $a_p\ne 0$. For $\alpha >0$ and $c\in\mR$, Kumar and Shukla [9] defined this nonlinear integral operator on ${\mathfrak{A}}(p)$:
\[ T\big(f(z);p, \alpha, c\big)=\left[ \frac{c+p\alpha}{z^c}\int^z_0\,\zeta^{c-1} f^{\alpha}(\zeta)\, d\zeta  \right]^{\frac{1}{\alpha}}.\tag{5.1} \label{eq:special}\]
This transform, related to Bazilevi${\breve{\rm{c}}}$ functions, includes the Bernardi transform as well as other transforms  previously studied by many authors. They defined the subclass of starlike functions
$$ S^*_p(A,B)=\left\{ f\in {\mathfrak{A}}(p):\, \frac{zf'(z)}{f(z)}\prec p\left(\frac{1+Az}{1+Bz}\right),\,-1\le B<A\le 1\right\}$$  and proved that  $T\left(S^*_p(A,B);p, \alpha, c\right)\subset S^*_p(A,B)$ whenever $c\ge -p\alpha \left(\frac{1-A}{1-B}\right)$.

We can prove the following general result on invariant subclasses with respect to this transform, which includes their theorem when $\alpha$ and $c$ are integers:

\medskip

\begin{theorem}
 Let $\alpha, p\in\mN$ and $c=0, 1, 2, \cdots$ and let $\Psi (z)$ be a convex univalent function in ${\mathcal{H}}(\mD)$ with $\Psi(0)=p$ and $ \ds\Re e \,\big\{\Psi(z)\big\}>-\frac{c}{\alpha}$, \, for $z\in\mD$. If  $${\mathcal{F}}=\left\{f\in \mathfrak{A}(p):\, \frac{zf'(z)}{f(z)}\prec \Psi(z)\right\},$$
then $T\Big({\mathcal{F}};p, \alpha, c\Big)\subset {\mathcal{F}}$.
\end{theorem}
\medskip

\begin{proof}
Let $f\in {\mathcal{F}}$, then $f$ is starlike. For convenience let 
$$F(z)=\left[ \frac{c+p\alpha}{z^c}\int^z_0\,\zeta^{c-1} f^{\alpha}(\zeta)\, d\zeta  \right]^{\frac{1}{\alpha}}.$$
Set $g(z)=z^cf^{\alpha}(z)$ and $k(z)=\ds \int^z_0\,\zeta^{c-1} f^{\alpha}(\zeta)\, d\zeta$. It follows that 
$$\frac{zg'(z)}{g(z)}=\alpha\left(\frac{zf'(z)}{f(z)}\right)+c\prec \alpha \,\Psi(z)+c,$$
thus $g$ is starlike and since $f$ is $p-$valent, the function $g$ is $(\alpha p+c)-$valent.

Now it follows that 
$$\frac{k'(z)}{g'(z)}=\frac{1}{\alpha\left(\frac{zf'(z)}{f(z)}\right)+c}\prec \left[\alpha \,\Psi(z)+c\right]^{-1}$$
and hence $\ds \Re e\left\{\frac{k'(z)}{g'(z)}\right\}>0, k(0)=g(0)=0$ and $\ds\frac{k(0)}{g(0)}=\frac{k'(0)}{g'(0)}=\frac{1}{\alpha p+c}$. Apply Lemma 2.1 with $\ds \Phi(z)=\frac{1}{\alpha p+c}\left(\frac{1+z}{1-z}\right)$ to obtain $\ds \Re e\left\{\frac{k(z)}{g(z)}\right\}>0$ and since $\ds\frac{k(z)}{g(z)}=\frac{k(z)}{zk'(z)}$ we have $\ds \Re e\left\{\frac{k(z)}{zk'(z)}\right\}>0$. Thus $k(z)=\ds \int^z_0\,\zeta^{c-1} f^{\alpha}(\zeta)\, d\zeta$ is starlike and $(\alpha p+c)$-valent since
$$\frac{1}{2\pi}\int^{2\pi}_0 \Re e\left\{\frac{re^{i\theta}k'(re^{i\theta})}{k(re^{i\theta})}\right\}=\frac{g(0)}{k(0)}=\alpha p+c.$$

Next, this time we set  $k(z)=z^cf^{\alpha}(z)$ and $g(z)=\ds \int^z_0\,\zeta^{c-1} f^{\alpha}(\zeta)\, d\zeta$ and by the above, both  $g$ and $k$ are starlike and $(\alpha p+c)$-valent, $\ds k(0)=g(0)=0, \frac{k(0)}{g(0)}=\frac{k'(0)}{g'(0)}=(\alpha p+c)$. A calculation gives 
$$\frac{k'(z)}{g'(z)}=\alpha\left(\frac{zf'(z)}{f(z)}\right)+c\prec \alpha \Psi(z)+c$$
and apply Lemma 2.1 with $\Phi(z)=\big(\alpha \Psi(z)+c\big)$ to conclude that
$$\alpha\left(\frac{zF'(z)}{F(z)}\right)+c=\frac{k(z)}{g(z)}=\prec \alpha \Psi(z)+c.$$
Hence $F\in {\mathcal{F}}$.

\end{proof}

\noindent  The result of Kumar and Shukla (for $\alpha\in\mN, c=0, 1, 2, \cdots$) now follows if we simply choose the convex function $\ds \Psi(z)=p\left(\frac{1+Az}{1+Bz}\right)$.
\medskip

Our method may be applied to other classes of functions as follows: 
\begin{theorem} Let $\ds {\mathcal{F}}=\left\{f\in H_0(\mD):\, e^{i\beta}\frac{zf'(z)}{f(z)}\prec \Psi(z)\right\}$, 
where $\Psi$ is convex, $\Psi(0)=e^{i\beta}$, and $\ds \Re e \, \Psi(z)>0$ for all $z\in\mD$. If $\ds |\beta|<\frac{\pi}{2}$, then $T\left({\mathcal{F}};p, \alpha, c   \right)\subset {\mathcal{F}}$. That is, $T\left(f;p, \alpha, c   \right)$ preserves ${\mathcal{F}}$.
\end{theorem}

\nni If we let $\ds \Psi(z)=e^{i\beta}\left(\frac{1-z}{1+ze^{2i\beta}}\right)$, then  ${\mathcal{F}}$ is the  class of univalent spirallike functions studied by [21, 8, 12]. The above theorem is apparently the first of its kind for such functions.

There are various other transforms that are also amenable to our techniques. For instance for $\beta\in\mC$, Kim and Merkes [8] introduced  the transform  $$\ds \left(I_{\beta}f\right)(z)=\int^z_0\left(\frac{f(\zeta)}{\zeta}\right)^{\beta} d\zeta$$ and proved that $I_{\beta}$ maps $S$ into $S$, provided $|\beta|\le \frac{1}{4}$. It is readily shown that a result analogous to Theorem 5.1 is also true for this new transform:

\medskip

\begin{theorem}
 Let $\beta\in\mC$ and let $\Psi (z)$ be a convex univalent function in ${\mathcal{H}}(\mD)$ with $\Psi(0)=1$ and $\ds\Re e \Big\{\beta \Psi(z)+(1-\beta\Big\}>0$ \, for $z\in\mD$. If  $${\mathcal{G}}=\left\{f\in \mathfrak{A}(p):\, \frac{zf'(z)}{f(z)}\prec \Psi(z)\right\},$$
then  $I_{\beta}\Big({\mathcal{G}}\Big)\subset {\mathcal{G}}$.
\end{theorem}
\medskip
\nni If, for example, we set $p=1$ and $\ds \Psi(z)=\frac{1+z}{1-z}$, this leads to new and rather curious results including:

\begin{corollary} If $f\in S^*$ and $\Re e\,  \{\beta \}<1$, then $\ds\int^z_0\left(\frac{f(\zeta)}{\zeta}\right)^{\beta} d\zeta\in S^*$.
\end{corollary}

\medskip

\nni Even for transforms like $\ds P_{\lambda}(f)(z)=\int^z_0 f'(\zeta)^{\lambda}\,d\zeta$, introduced by Pfaltzgraff [18], we can prove that if $f\in {\mathcal{H}}_0(\mD)$ and 
$$\Re e\,\left\{1+\frac{\lambda zf''(z)}{f'(z)}\right\}>0$$
 for all $z\in \mD$, then $\ds P_{\lambda}(f)\in S^*$.

\bibliographystyle{unsrt}

\end{document}